\documentclass[11pt]{amsart}

\usepackage{amsmath,amssymb,amscd,amsfonts,amsthm,verbatim}
\usepackage[all,cmtip]{xy}
\usepackage{paralist}
\usepackage[mathscr]{eucal}
\usepackage{color}
\usepackage{pdfsync}
\usepackage[]{fontenc}
\usepackage{enumerate}
\usepackage{extarrows}

\usepackage[pdftex]{hyperref}
\hypersetup{citecolor=blue,linktocpage}

\usepackage[colorinlistoftodos]{todonotes}
\newtheorem{thm}{Theorem}[section]
\newtheorem{lem}[thm]{Lemma}
\newtheorem{prop}[thm]{Proposition}
\newtheorem{cor}[thm]{Corollary}

\newtheorem{assu-nota}[thm]{Assumption--Notation}

\theoremstyle{definition}
\newtheorem{defn}[thm]{Definition}
\newtheorem{nota}[thm]{Notation}
\newtheorem{rem}[thm]{Remark}
\newtheorem{ex}[thm]{Example}

\newcommand{\inv}{^{-1}}

\newcommand{\iso}{\cong}
\newcommand{\into}{\hookrightarrow}
\newcommand\onto{\twoheadrightarrow}

\newcommand{\A}{\mathbb A}

\newcommand{\K}{\mathbb K}

\newcommand{\bL}{\mathbb L}

\newcommand{\pp}{\mathbb P}

\newcommand{\sing}{_{\rm sing}}
\newcommand{\red}{_{\rm red}}

\newcommand{\OO}{\mathcal O}

\newcommand{\cE}{\mathcal E}

\newcommand{\cG}{\mathcal G}
\newcommand{\cL}{\mathcal L}

\newcommand{\cT}{\mathcal T}

\newcommand{\homc}{{\mathcal Hom}}
\newcommand{\ext}{{\mathcal Ext}}
\DeclareMathOperator{\Ann}{Ann}

\DeclareMathOperator{\Pic}{Pic}

\DeclareMathOperator{\Proj}{Proj}
\DeclareMathOperator{\Spec}{Spec}

\DeclareMathOperator{\car}{char}

\DeclareMathOperator{\Id}{Id}

\DeclareMathOperator{\Tr}{Tr}

\DeclareMathOperator{\coker}{coker}

\DeclareMathOperator{\sym}{Sym}
\DeclareMathOperator{\Tor}{Tor}
\DeclareMathOperator{\Qcoh}{Qcoh}
\DeclareMathOperator{\Coh}{Coh}

\newcommand{\wt}{\widetilde}

\numberwithin{equation}{section}

\title{Deformations  of semi-smooth varieties}
\author{Barbara Fantechi, Marco Franciosi and Rita Pardini}

\begin{document}
\maketitle

\begin{abstract} For a singular variety $X$, an essential step to determine its smoothability and study its deformations is the understanding of
 the tangent sheaf and of the sheaf $\mathcal T^1_X:=\ext^1(\Omega_X,\mathcal O_X)$. \\
A variety is semi-smooth if its singularities are \'etale locally the product of a double crossing point ($uv=0$) or a pinch point ($u^2-v^2w=0$) with affine space; equivalently, if it can be obtained by gluing a smooth variety along a smooth divisor via an involution with smooth quotient.\\ 
Our main result is the explicit computation of the tangent sheaf and the sheaf $\mathcal T^1_X$ for a semi-smooth variety $X$ in terms of the gluing data.   


\par
\medskip
\noindent{\em keywords:} semi-smooth, infinitesimal deformations, push-out scheme

\end{abstract}

\setcounter{tocdepth}{1}
\tableofcontents

\section{Introduction}

For $X$ a singular projective variety, it is natural to ask whether it can be smoothed in a flat proper (or projective) family.
 A first necessary condition is the nonvanishing of the space of global sections of the sheaf $\mathcal T^1_X:=\ext^1(\Omega_X,\mathcal O_X)$; 
 in fact, if $H^0(X,\mathcal T^1_X)=0$, then all infinitesimal deformations of $X$ are locally trivial, and in particular preserve the singularities 
 (see \cite{Schlessinger}).

Sufficient conditions are more difficult to obtain, especially if we assume that the singularities are non isolated.  A classical result of Friedman (cf. \cite{Friedman})  shows smoothability for varieties with simple normal crossings under some very special conditions.
Such results has been recently generalized by Felten,  Filip and Ruddat (cf. \cite{ruddat}) in the realm of toroidal local models.

In  \cite{tziolas2010}, Tziolas proves that if we assume that $X$ has lci singularities then a formal smoothing exists, provided that $\mathcal T^1_X$ is generated by global sections and that $H^1(X,\mathcal T^1_X)=H^2(X,T_X)=0$ (in this case the deformations are also unobstructed).

As this result shows, it is important to compute explicitly the sheaves $T_X$ and $\mathcal T^1_X$.
In this paper we do so for semi-smooth varieties, a class of singularities that naturally appear on stable surfaces in the boundary of the moduli of surfaces of general type. 

A surface is semi-smooth if its only singularities are double crossings and pinch points (see e.g. Def.~4.1 in \cite{KSB}); in Definition \ref{def: semi-smooth} we call a variety $X$ {\em semi-smooth} if it is \'etale locally the product of a semi-smooth surface with affine space. 

As a first step we show that a   variety  $X$  is semi-smooth if and only if it is  a push-out scheme 
\begin{equation}\label{eq: diag0}
\begin{CD}
\bar Y @>g>>Y\\
@V {\bar \jmath} VV @VV j V\\
\bar X@>f >>X
\end{CD}
\end{equation}
where  $\bar X, \bar Y$ and  $Y$ are smooth varieties, $\bar Y\to \bar X$ is a codimension  1 closed embedding, and  $\bar Y\to Y$  is a double cover  (see Section \ref{sec: semi-smooth}).
This is consistent with Koll\'ar's philosophy (see \cite{kollar-sings}) of describing slc varieties in terms of their associated lc pairs $(\bar X, \bar Y)$ and gluing involution on the normalization of $\bar Y$.

%In the Appendix, we show that this is equivalent to $X$ being obtained by
%gluing a smooth variety $\bar X$ (the normalization of $X$) along a smooth divisor $\bar Y$ via an involution $\iota$ with smooth quotient.

Our  first main result  is   the explicit computation of the sheaf $T_X$ using the above description 
(see  Theorem \ref{prop: tg}):

\begin{thm} In the above setup: 
\begin{enumerate}
\item there is a natural injective map $\alpha\colon  T_X\to f_*T_{\bar X}$ which is an isomorphism on  the smooth locus of $X$;

\item set $\mathcal G:=\coker \alpha$;  then $\alpha$ induces   an exact sequence
$$0\to (g_*T_{\bar  Y})^{\rm inv}\to g_*T_{\bar  X}|_{\bar  Y}\to \mathcal G\to 0.$$
\end{enumerate}
\end{thm}

\medskip

The double cover $\bar Y\to Y $ is completely determined by an isomorphism $ L^{\otimes 2} \simeq  \OO_Y ( D)$, where $D\subset Y$  is  the  (necessarily smooth) branch divisor  and  $L\in \Pic(Y)$  (see Section \ref{ssec: double-def}).  
The sheaf ${\mathcal T}^1_X$ is an invertible sheaf on the scheme theoretic singular locus of $X$; it is an extension of a line bundle on $Y$ by a line bundle on $D$:
$$
0\to \mathcal T^1_X\otimes \mathcal I_{Y|X\sing} \to \mathcal T^1_X\to \mathcal T^1_X|_Y\to 0. 
$$
Our second main result is the computation of the first and last term in the above exact sequence
 (see Theorem \ref{thm: computing-T1}):

\begin{thm} We have the following isomorphisms of line bundles: 

\begin{enumerate}
 \item  on $Y$, $\cT^1_X|_Y\cong L\otimes (\det g_*(N_{\bar Y|\bar X}\inv ))\inv$;
\item on $\bar Y$, $g^*(\cT^1_X|_Y)\cong   g^*(L^{\otimes 2})\otimes N_{\bar Y|\bar X}\otimes \iota^*N_{\bar Y|\bar X}$;
\item  on  $D$, $ I_{Y|X\sing}\cong   r_* \left( (N_{\bar Y|\bar X}\inv)_{|R} \right)$, where $r\colon R\to D$ is the natural  isomorphism from the ramification to  the branch divisor of $g$.
\end{enumerate}
\end{thm}

These results  are applied in \cite{FFP19b} to prove the smoothability of all semi-smooth singular stable Godeaux surfaces,
classified in \cite{FPR18}; we expect that the techniques developed here will also apply to other classes of varieties with 
hypersurface singularities with smooth normalization.

Our methods combine different approaches, leading us to obtain along the way results of independent interest, and in greater generality than strictly needed here. 

In section 2 we prove that the relative version of the sheaf $\mathcal T^1$ commutes with specialization for flat families of lci varieties, by making its relationship with the cotangent complex explicit. In section 3 we recall basic results on gluing schemes and give the characterization of semi-smooth varieties via gluing. In section 4 we describe explicitly $X$ as a hypersurface in a rank 2 vector bundle over $Y$ when $\bar X$ is the total space of a line bundle on $\bar Y$; using this we compute explicitly $\mathcal T^1_X$. 

In section 5 we prove the two main theorems; the sheaf $T_X$ is computed as pushforward from a sheaf on $\bar X$ by a mix of global constructions and \'etale local computations; we reduce the computation of the sheaf $\mathcal T^1_X$ to the special case in section 4 by deforming to the normal cone the closed embedding of $\bar Y$ in $\bar X$ and applying the specialization result proved in section 2.

\medskip

\subsection*{Acknowledgements}
This article is based upon work supported by the National Science Foundation under
 Grant No. 1440140, while the first and third author were in residence at the Mathematical Sciences Research Institute in Berkeley, California, 
 during the spring semester of 2019.  
 This collaboration  started during the workshops ``Derived Algebraic Geometry and Birational Geometry and Moduli Spaces'' and
  ``Connections for Women: Derived Algebraic Geometry'',  January 2019, MSRI Berkeley. 

This project was partially supported by the projects PRIN  2015EYPTSB$\_$ 010 ``Geometry of Algebraic Varieties" and PRIN 
 2017SSNZAW$\_$004 ``Moduli Theory and Birational Classification"  of Italian MIUR.
 All authors are members of GNSAGA of INDAM.  
 
 We thank Alice Rizzardo for pointing out reference \cite{Riz}. 

\medskip
\hfill\break
{\bf Notation and Conventions.}  
All  schemes are assumed to be Noetherian and such that  $\frac 12\in \OO_X$.
Varieties are equidimensional reduced schemes of finite type over  an algebraically closed field  $\K$ with $\car \K\ne 2$; 
when talking about points of a variety we restrict our attention to closed points.
 If needed, we state additional assumptions at the beginning of sections.\par

 For a  vector bundle $E$  on a  scheme $X$, 
we follow the conventions of \cite{Hartshorne}  and we write  $V_X(E):=\Spec(\sym E)$ and $\pp_X(E):=\Proj(\sym E)$; 
we will drop the subscript $X$ when no confusion is likely to arise. We identify invertible sheaves and Cartier divisors and we use the additive and multiplicative notation interchangeably. Linear equivalence is denoted by $\sim$.

\section{The sheaf $\mathcal T^1$ for lci varieties and flat lci morphisms}

In the first two subsections we summarize known material for the reader's convenience.
In subsection \ref{ssec: glue-ext} we prove a specialization result which is crucial for the results of section \ref{sec: def}. 

\subsection{Definitions and local properties}

\begin{nota} Let $\pi\colon Y\to B$ be a flat morphism: we denote by  $T_{\pi}$ the sheaf $\homc(\Omega_\pi,\mathcal O_Y)$ and by  $\mathcal T^1_\pi$ the sheaf $\ext^1(\Omega_\pi,\mathcal O_Y)$. 
If $B=\Spec \K$ then we write $T_Y$, $\mathcal T^1_Y$ instead of $T_{\pi}$,  $\mathcal T^1_\pi$. \end{nota}

\begin{defn} We  say that a morphism of schemes is  {\em lci} (``locally complete intersection'') if it is  of finite type and it factors locally as a (closed) regular embedding followed by a smooth morphism, both of finite type. \end{defn}

Note that this differs slightly from the use in \cite{Fulton} where lci means that there exists a global such factorization.

\begin{rem} Let $q\colon Z\to W$ be an lci morphism of schemes, and assume that it factors as $s\circ i$ where $i\colon Z\to M$ is a regular closed embedding and $s\colon M\to Z$ is a 
smooth morphism; let $\mathcal I$ be the ideal sheaf of $X$ in $M$. Then there is an exact sequence of coherent sheaves on $X$ $$
i^*\mathcal I\to i^*\Omega_s\to \Omega_q\to 0$$
with $i^*\mathcal I$ and $i^*\Omega_s$ locally free (\cite{Fulton} B.6.1).
Over the locus  in $Z$ where $q$ is smooth, the sequence is also exact on the left; it follows, if this locus is dense 
(e.g. if $q$ is flat and the fibers of $q$ are generically smooth), 
that the sequence is exact on the left, thus providing a locally free resolution of $\Omega_q$. 
In particular, we get an induced exact sequence $$
0\to T_q\to i^*T_s \to \mathcal N_i \to \mathcal T^1_q\to 0,$$
where $\mathcal N_i=(i^*\mathcal I)^\vee$ is the normal bundle of $Z$ in $M$.
\end{rem}
\begin{rem} \label{rem: T1-divisor}
Assume moreover that $i\colon Z\to M$ is a codimension $1$ regular embedding, i.e., $Z$ is an effective Cartier divisor in $M$. 
Then  $\mathcal T^1_q$ is a quotient of the invertible sheaf $\mathcal N_i=i^*\mathcal O_M(Z)$ on $M$, thus it is a line bundle on a uniquely defined closed subscheme of $Z$;
 this  is a natural closed subscheme structure on the locus of points where $q$ is not smooth.\end{rem}

\begin{defn}\label{def: SIng} We say that a flat lci morphism $q\colon Z\to W$ with generically smooth fibers is {\em locally hypersurface} if it locally admits a factorization as in the previous remark.
 Again it follows that $\mathcal T^1_q$ is an invertible sheaf on a closed subscheme, the {\em singular locus} $Z_{q, { \rm sing}}$.
  If $W=\Spec \K$ then  we drop $q$ from the notation.
\end{defn}

\begin{ex}
 If $M=\Spec R$ and $W=\Spec \K$ is affine with local coordinates $x_1,\ldots,x_{n+1}$ and $f\in R$ is an equation for $Z$, then
 the ideal of $Z\sing$,  the associated Jacobian ideal,  is generated by $\partial f/\partial x_1,\ldots, \partial f/\partial x_{n+1}$.

\end{ex}

\subsection{Relationship with the cotangent complex}

In order to study the case of an lci (or hypersurface) morphism $q\colon Z\to W$ which may not admit a factorization as a regular embedding (or effective Cartier divisor) followed by a smooth morphism, it is useful to relate this notion to that of cotangent complex; this will play a key role in the proof of Theorem \ref{thm: restriction}.

 \begin{nota} If $X$ is a scheme, we  denote 
 by $D(X)$ the derived category of sheaves of $\mathcal O_X$-modules. 
 If $\mathcal A$ is a sheaf of  $\mathcal O_X$-modules, we   denote  by $\mathcal A^c$ the complex in  $D(X)$  that has the sheaf $\mathcal A$ in degree zero, and zero in all other degrees. 
\end{nota}
For any morphism of schemes $q\colon Z\to W$, we denote by  $\mathbb L_q\in D(Z)$
its {\em cotangent complex}; it has zero cohomology in every positive degree  and  $h^0(\mathbb L_q)$ is canonically isomorphic to $\Omega_q$
 (see e.g. \cite[\href{https://stacks.math.columbia.edu/tag/08UQ}{Tag 08UQ}]{stacks}). 

\begin{rem}\label{rem: cotg-lci}
Let $q\colon Z\to W$ be an   lci morphism; then the cotangent complex   $\bL_q$  is perfect of tor amplitude in $[-1,0]$.  In fact, this is a local property (\cite[\href{https://stacks.math.columbia.edu/tag/08T1}{Tag 08T1}]{stacks}),  so we may assume all schemes are affine, and it holds in the affine case by Thm.~5.4 and Corollary 6.14 of   \cite{Quillen} (see also  \cite[\href{https://stacks.math.columbia.edu/tag/08SH}{Tag 08SH}]{stacks}).\end{rem}

\begin{rem} In particular, if $q$ admits a global factorization as a regular embedding $i$ with ideal sheaf $\mathcal I$ followed by a smooth morphism $s$, then there is a canonical isomorphism in $D(Z)$ between $\mathbb L_q$ and the complex of locally free sheaves $[i^*\mathcal I\to i^*\Omega_s]$ in degree $[-1,0]$.
\end{rem}

The cohomology sheaf  $h^{-1}(\bL_q)$ is locally a subsheaf of a free sheaf, hence torsion free. If in addition $q$ is flat and  the fibers of $q$ are generically smooth, the locus in $Z$  where $q$ is smooth is dense, but on the smooth locus $h^{-1}(\bL_q)=0$. So  the only nonzero cohomology sheaf of $\bL_q$  is $h^0(\bL_q)=\Omega_q$, hence $\bL_q$ is canonically isomorphic to $(\Omega_q)^c$ in $D(Z)$.

Recall that given $E$ and $F$  quasicoherent sheaves on a scheme $Z$, for every $i\ge 0$  there is a natural isomorphism in $\Qcoh(Z)$ $$
h^iR\homc(E^c,F^c)\cong \ext^i(E,F).$$

\begin{cor} If $q\colon Z\to W$ is a flat lci morphism with generically smooth fibers, then  $\mathcal T^1_q$ is canonically isomorphic to $h^1R\homc(\bL_q, \OO_Z^c)$. 
\end{cor}

 \subsection{Base change for  ${\mathcal T}^1$}\label{ssec: glue-ext}

A key step for the computations of \S \ref{ssec: T1} will be the fact that under suitable assumptions  the sheaf $\mathcal T^1$ is stable under base change by a closed embedding.

 \begin{thm}\label{thm: restriction} Let $\pi\colon Y\to B$ be a flat lci morphism of schemes (of finite type over $\K$) with generically smooth (e.g.,  reduced) fibers, $g\colon  C\to B$ a closed embedding of schemes; let $X:=Y\times_{B}C$ and denote by $f\colon X\to Y$ and $p\colon X\to C$ the projection maps, yielding the following cartesian diagram: \[
\begin{CD}
X @>{f}>>{Y}\\
@V{p}VV @VV{\pi}V\\
C@>>{g}> B.
\end{CD}
\]
Then there is a natural isomorphism $$f^*({\mathcal T}^1_{\pi})\cong {\mathcal T}^1_p \text{ in } \Coh(X).$$
 \end{thm}

The strategy of the proof is to construct the claimed isomorphism globally, then prove that it is an isomorphism locally. In order to do this we make use of the properties of the cotangent complex. 

 \begin{rem}\label{rem: base-change-lci} 
Consider  a Cartesian diagram 
\[
\begin{CD}
X @>{f}>>{Y}\\
@V{p}VV @VV{\pi}V\\
C@>>{g}> B
\end{CD}
\]
If $\pi$ is a flat morphism, then there is a natural isomorphism $Lf^*\bL_\pi\to \bL_p$ in $D(X)$ by \cite[Tag 09DJ]{stacks}. Moreover  if in addition  $\pi$ is  lci,  then $p$ is also an lci morphism by \cite{Fulton} Proposition 6.5(a).
\end{rem}

\begin{lem}\label{lem:proj_form} Let $f\colon X\to Y$ be a morphism of schemes, $E\in D(Y)$ and $F\in D(X)$. Then there is a natural, functorial isomorphism in $D(Y)$
$$Rf_*(R\homc(Lf^*E,F))\cong R\homc(E,Rf_*F)$$
which commutes with restriction to open subsets.
\end{lem}
\begin{proof} See for instance Lemma 2.1 in \cite{Riz} and references therein. 
\end{proof}

\begin{prop}\label{prop:MorExt} In the assumptions of Theorem \ref{thm: restriction}, there is a natural homomorphism
$$f^*({\mathcal T}^1_\pi)\to {\mathcal T}^1_p.$$
\end{prop}
\begin{proof} 
We will construct a natural homomorphism ${\mathcal T}^1_\pi\to f_*({\mathcal T}^1_p)$; since  $f^*$ is left adjoint to $f_*$, this will  prove the claim.
\smallskip

\noindent Lemma \ref{lem:proj_form} applied to $E=\bL_\pi$ and $F=\OO^c_X$ yields a natural isomorphism 
$Rf_*R\homc(Lf^*\bL_\pi, \OO^c_X)\cong R\homc(\bL_\pi,Rf_*(\OO^c_X))$ 
in the derived category $D(Y)$. 

Since $f$ is a closed embedding, $f_*$ is exact, hence $Rf_*(\OO_X^c)=(f_*\OO_X)^c$; so by Remark \ref{rem: base-change-lci} we obtain a  a natural isomorphism $$Rf_*R\homc(\bL_p, \OO_X^c)\cong R\homc(\bL_\pi,(f_*\OO_X)^c).$$
This induces isomorphisms on cohomology sheaves: in particular, $$
h^1(Rf_*R\homc(\bL_p, \OO^c_X))\cong h^1(R\homc(\bL_\pi,(f_*\OO_X)^c).$$
Since $f_*$ is exact because $f$ is a closed embedding, $$h^1(Rf_*R\homc(\bL_p, \OO^c_X))=f_*(h^1(R\homc(\bL_p, \OO_X^c))).$$
By Remark \ref{rem: cotg-lci}  we obtain 
$$h^1(Rf_*R\homc(\bL_p, \OO^c_X))\cong f_*({\mathcal T}^1_p).$$
Similarly, we have 
$$h^1R\homc(\bL_\pi,Rf_*(\OO^c_X)) \cong h^1R\homc(\bL_\pi,(f_*\OO_X)^c)\cong \ext^1(\Omega_\pi,f_*\OO_X).$$
Putting everything together we get a natural isomorphism of sheaves $$
\ext^1(\Omega_\pi,f_*\OO_X)\cong f_*({\mathcal T}^1_p).$$
It is now enough to compose  this isomorphism with the homomorphism ${\mathcal T}^1_{\pi} \to   \ext^1(\Omega_\pi,f_*\OO_X)$ induced by the structure morphism $\OO_Y\to f_*\OO_X$ to obtain the desired  morphism $$
{\mathcal T}^1_{\pi} \to f_*({\mathcal T}^1_p) \in \Qcoh(Y).$$
\end{proof}

\hfill\break 
{\em Proof of Theorem~\ref{thm: restriction}}.
We need to show that the morphism
 $$
f^*{\mathcal T}^1_{\pi}\to ({\mathcal T}^1_p)$$ defined in Proposition \ref{prop:MorExt} is an isomorphism. 
This is a local property, so we may assume that $Y$, $B$ and $C$ (and hence $X$) are affine.
 Consider  a locally free,  finite rank resolution in $D(X)$ of $\Omega_{\pi}$, which exists in view of  Remark \ref{rem: cotg-lci}  because $Y$ is affine:
 \[
\begin{CD} 0 @>>> \cE^{-1} @>{\varphi}>> \cE^0 @>>> \Omega_{\pi}  @>>> 0 
\end{CD}
\]
This implies that $\bL_{\pi}\iso [ \cE^{-1}\to \cE^0 ]$ in $D(Y)$. By Remark \ref{rem: base-change-lci} it follows that $Lf^*[ \cE^{-1}\to \cE^0 ]\iso \bL_p$ in $D(X)$; since $\cE^i$ are locally free, they are flat over $B$, thus $Lf^*[ \cE^{-1}\to \cE^0 ]=[ f^*\cE^{-1}\to f^*\cE^0 ]$. Hence, again by  Remark \ref{rem: base-change-lci} we have an exact sequence $$
\begin{CD} 0 @>>>f^* \cE^{-1} @>{f^*\varphi}>> f^*\cE^0 @>>> \Omega_{p}  @>>> 0. 
\end{CD}
$$

Dualizing the above sequences (on $Y$ and $X$, respectively) we get 
 \[
 \begin{CD}
 (\cE^{0})^{\vee} @>\varphi^{\vee}>>  (\cE^{-1})^{\vee}  @>>>{\mathcal T}^1_{\pi} @>>> 0 
\end{CD}
\]
 \[\begin{CD}
 f^*(\cE^{0})^{\vee} @>(f^*\varphi)^{\vee}>>  f^*(\cE^{-1})^{\vee}  @>>> {\mathcal T}^1_p  @>>> 0 
\end{CD}
\]

Applying $f^*$ to the first sequence and comparing yields the isomorphism $f^*({\mathcal T}^1_{\pi} )\to {\mathcal T}^1_p$. This completes the proof.

\begin{cor}\label{cor: sing_pullback}  In the assumption of Theorem 
\ref{thm: restriction}, if $\pi$ is of hypersurface type, then so is $p$ and the scheme theoretic intersection of $X$ with $Y_{\pi, {\rm sing}}$ is $X_{p, {\rm sing}}$.
\end{cor}

\begin{proof} The map $p$ is of hypersurface type since the map $\pi$  is of hypersurface type and flat. 
The result on the singular locus follows from Theorem 
\ref{thm: restriction} and Definition \ref{def: SIng}.
\end{proof}

\section{Gluing a scheme along a closed subscheme}\label{sec: glue}

In this section we   
study in detail  the properties of the scheme $X=\bar X\sqcup_{\bar Y}Y$ 
 obtained  by gluing  a scheme  $\bar X$ via a finite morphism  $g\colon \bar Y\to Y$, where $\bar Y \subset \bar X$ is a closed subscheme.

\subsection{Generalities}\label{ssec: generalities}

In this subsection we recall briefly what it means to glue (pinch) a scheme along a closed subscheme. We  mainly follow \cite{Ferrand}, that works in the category of schemes; more general situations, leading to the construction of algebraic spaces, are considered for instance  in \cite{kollar-quotients}, \cite[Ch.~5 and 9]{kollar-sings}. 
\bigskip

Fix a  Noetherian scheme $S$; in this section all schemes are $S$-schemes and maps are maps of $S$-schemes.   We recall from \cite{Ferrand} the following:
\begin{defn}\label{def: AF}
We say that a scheme $X$ satisfies property $(AF)$ (or Chevalley-Kleiman property, cf.~\cite{kollar-quotients}) if every finite subset of $X$ is contained in an affine open  set. Note that a scheme satisfying $(AF)$ is separated,  since $X\times X$ can be covered by  open sets of the form $U\times U$ with $U\subseteq X$  open affine.
\end{defn}

\begin{rem}\label{rem: AF}
 Quasi-projective varieties over an algebraically closed field satisfy property $(AF)$. Without loss of generality we may assume  that $X=Y\setminus Z\subseteq \A^N$, where $Y$ and $Z\subset Y$ and  are closed. 
Let   $p_1, \dots p_k\in X$ be distinct points. Clearly  it is enough to prove the claim when the $p_i$ are closed points. Then for every $i$ we can  find $\phi_i\in I(Z)\subseteq \K[\A^N]$ such that $\phi_i(p_i)=1$ and $\phi_i(p_j)=0$ if $i\ne j$. If we set $\phi:=\phi_1+\dots+\phi_k$ with $a_i\in \K$ general,  then $p_1,\dots p_k$ are contained in the affine open set $Y_{\phi}\subset X$.
\end{rem} 

We consider the following setup: $\bar X$ and $Y$ are schemes satisfying (AF) and  we are given a scheme $\bar Y$ with a closed embedding  $\bar \jmath \colon \bar Y\to \bar X$   and  a finite morphism $g\colon \bar Y\to Y$. Then we have the following:

\begin{thm}[Ferrand, Thm.~5.4 of \cite{Ferrand}]\label{thm: glue}
If $\bar X$ and $Y$ satisfy property $(AF)$, then there exists a scheme $X$ also satisfying $(AF)$  fitting in the  following cocartesian diagram: 
\begin{equation}\label{eq: diag1}
\begin{CD}
\bar Y @>g>>Y\\
@V {\bar \jmath} VV @VV j V\\
\bar X@>f >>X
\end{CD}
\end{equation}
such that:
\begin{itemize}
\item[(a)] diagram \eqref{eq: diag1} is cartesian 
\item[(b)] $f$ is finite  and $j$ is a closed embedding
\item[(c)] $f$ restricts to an isomorphism $\bar X\setminus \bar Y\to X\setminus Y$.
\end{itemize}
\end{thm}

The scheme $X$  whose existence is given in  Theorem \ref{thm: glue} is called the {\em push-out scheme  obtained from $\bar X$ by gluing (pinching) $\bar X$ along $\bar Y$ via $g$}; following \cite{Ferrand}, we   often write $X=\bar X\sqcup_{\bar Y}Y$. 

\begin{rem}\label{rem: glue-local}
 Theorem  \ref{thm: glue} is proven by considering the affine case first and then showing that the construction globalizes. 
In the affine case  $S=\Spec R$,  $\bar X=\Spec \bar A$, $\bar Y=\Spec \bar B$, where $\bar B=\bar A/I $ and $Y =\Spec B$,  one has $X= \Spec A$, where $A:= \bar A \times _{\bar A/I}B$.
By \cite[Thm.~41]{kollar-quotients}, if $\bar A$ is a finitely generated $R$-algebra, so is $A$. It follows that if $\bar X$ is of finite type over $S$, so is $X$.
\end{rem}

\begin{rem} \label{rem: cartesian}
It is well known (cf.~\cite{Ferrand}, Lemme 1.2) that a diagram:
\begin{equation}\label{diag: basic}
\begin{CD}
A@>\psi>>\bar A \\
@V p VV @VV{\bar p}V\\
B @>>> \bar B
\end{CD}
\end{equation}
where $p$ and $\bar p$ are surjective is cartesian if and only if $\psi$ maps $I:=\ker p$ isomorphically onto $\ker \bar p$. 
\end{rem}
\begin{rem}\label{rem: flat-base-change}
By Remark \ref{rem: cartesian},  if  diagram \eqref{diag: basic} is cartesian and we tensor it with a  flat $A$-algebra $R$ the resulting diagram is also cartesian; so by Remark \ref{rem: glue-local} if we take base change  of a pushout diagram as \eqref{eq: diag1} by a flat map $Z\to X$, the resulting diagram is cocartesian. \end{rem}
The following example describes the local situation we are interested in: 
\begin{ex}\label{ex: pinch}
Let $S=\Spec\K$, $\bar X=\A^2_{x,y}$, $\bar Y=\{y=0\}\subset \bar X$, $Y=\A^1_t$ and let   $g\colon \bar Y\to Y$ be the map given  by $(x,0)\mapsto x^2$. It is easy to check that $A=\K[x,y]\times_{\K[x]}\K[t]$ is generated as a $\K$-algebra by $u=( xy,0), v=(y, 0), w=(x^2,t)$. The generators satisfy the relation $u^2-v^2w=0$,  so $X$ is isomorphic to the hypersurface $\{u^2-v^2w=0\}\subset \A^3_{u,v,w}$. The singular point $(0,0,0)\in X$ is called a {\em pinch point}.

A similar (simpler) situation is the following: $\bar X=\{z^2-1=0\}\subset \A^3_{x,y,z}$, $\bar Y= \{z^2-1=y=0\}$, $Y=\A^1_t$ and   $g\colon \bar Y\to Y$ is defined by $(x, 0,z)\mapsto x$. Arguing as above we see that $X$ is isomorphic to $\{uv=0\}\subset \A^3_{u,v,w}$, $f\colon \bar X \to X$ is given by $(x,y,z)\mapsto (y(z-1),y(z+1),x)$ and $j\colon \A^1_t\to X$ is given by $t\mapsto(0,0,t)$.
\end{ex}

 \subsection{Semi-smooth varieties as push-out schemes}\label{sec: semi-smooth}
  \begin{defn}\label{def: semi-smooth}
An $n$-dimensional  variety $X$ over $\K$ is called {\em semi-smooth} if it is locally \'etale isomorphic\footnotemark\ to  $P_n:= \Spec\K[u,v,w]/(u^2-v^2w)\times\A^{n-2}$. 
\footnotetext{A variety $X$ is {\em locally \'etale isomorphic} to a variety $Y$ if $X$ can be covered by \'etale open sets isomorphic to \'etale open subsets of $Y$}
Points of $X$ corresponding to points of $\{u=v=w=0\}\subset P_n$ are called {\em pinch points};  points corresponding to points of $\{u=v=0, w\ne 0\}\subset P_n$ are {\em double crossings (dc) points}.

\end{defn}
\begin{rem}\label{rem: demi-normal}
 Semi-smooth varieties are locally complete intersections (lci), since the lci condition is local in the \'etale topology 
 (\cite[\href{https://stacks.math.columbia.edu/tag/08UQ}{Tag 06C3}]{stacks}). In particular, a semi-smooth variety  is $S_2$ and therefore  it is demi-normal (cf.~\cite[Def.~5.1]{kollar-sings}).
\end{rem}

\begin{rem}\label{rem: pinch}
Note that, in view of Example \ref{ex: pinch},  $P_n$ fits in the following cocartesian diagram:
\begin{equation}\label{diag: pinch}
\begin{CD}
\A^{n-1}_{x,t_1,\dots t_{n-2}} @>g>>\A^{n-1}_{w,t_1,\dots t_{n-2}} \\
@V {\bar{\jmath}} VV @VV j V\\
\A^n_{x,y,t_1,\dots t_{n-2}}@>f >>P_n
\end{CD}
\end{equation}
 where $\bar{\jmath} (x,t_1,\dots t_{n-2})=(x,0, t_1,\dots t_{n-2})$, $j(w,t_1,\dots t_{n-2})=(0,0, w,  t_1,\dots t_{n-2})$, $g(x, t_1,\dots t_{n-2})\mapsto (x^2, t_1,\dots t_{n-2})$. The map $f$ is given by $(x,y,t_1,\dots t_{n-2})\mapsto (xy,y,x^2, t_1,\dots t_{n-2}))$.
\end{rem}

Remark \ref{rem: pinch} suggests the following characterization of semi-smooth varieties, which we will  use systematically in \S \ref{sec: def}  to reduce computations to the situation  of diagram \eqref{diag: pinch}.

\begin{prop} \label{prop: semi-smooth-glue}
Let $X$ be an $n$-dimensional variety  over $\K$ that satisfies condition (AF).  Then the following  are equivalent:
\begin{enumerate}
\item $X$ is semi-smooth 
\item There exist a smooth variety $\bar X$,   a smooth divisor $\bar Y\subset X$ and a finite degree 2 map $g\colon \bar Y\to Y$ with $Y$ smooth such that $X=\bar X \sqcup_{\bar Y} Y$. 
\end{enumerate}
\end{prop}
 
 \begin{rem}\label{rem: semi-smooth-glue} In the situation of Proposition \ref{prop: semi-smooth-glue}, the variety $\bar X$ is the normalization of $X$ and $\bar Y$, resp. $Y$ are the subschemes of $\bar X$, resp. $X$ defined by the conductor (cf. \S \ref{ssec: A2}).
  The branch locus $D$ of $g$  is the set of pinch points of $X$.
 \end{rem} 
  Proposition \ref{prop: semi-smooth-glue} is very likely   well known to experts; for lack of a suitable reference we give the proof, which is a bit   lengthy,  in Appendix \ref{app: A}.
 
 For a semi-smooth variety $X=\bar X\sqcup_{ \bar Y} Y$ the support of the singular locus $X\sing$ (cf.~Definition  \ref{def: SIng}) is equal to $Y$ and $X_{\sing}$ is non reduced  at the pinch points of $X$; more precisely one has: 

\begin{lem}\label{lem: IXsing}
Let $X=\bar X\sqcup_{ \bar Y} Y$ be  a semi-smooth variety and let $D\subset Y$ be  the closed subset of pinch points of $X$;  then the ideal $I_{Y|X\sing}$ is an invertible sheaf on $D$. 

\end{lem}
\begin{proof}
It is enough to prove the claim for the pinch point $P_n$. In this case  the ideal of $Y$  in $\A^n_{u,v, w,t_1\dots t_{n-2}}$ is  $I_Y=(u,v)$ while the ideal of $X_{\sing}$ is $I_{X\sing}=(u,v^2, vw)$ and it is immediate to check that  $I_{Y|X\sing}=I_Y/ I_{X\sing}$ is supported on $D=\Spec\K[t_1,\dots t_{n-2}]$ and it is generated by the class of $v\in \K[t_1,\dots t_{n-2}]$. 
\end{proof}
\subsection{Gluing in families}\label{ssec: glue-in-family}
We show that the gluing construction of Theorem \ref{thm: glue} commutes with specialization under mild hypotheses: 
\begin{prop}\label{prop: special_glue} Notation and  assumptions as in \S \ref{ssec: generalities}.

Assume that  $\bar Y$ and  $Y$  are flat over $S$ and let $s\in S$ be a point. 
Then the diagram: 
\[
\begin{CD}
\bar Y_s @>{g|_{\bar Y_s}}>>Y_s\\
@V {{\bar \jmath }|_{Y_s}} VV @VV {j|_{ Y_s}}V\\
\bar X_s@>{f|_{\bar X_s}} >>X_s
\end{CD}
\]
is cocartesian.
\end{prop}
\begin{proof}
Since the construction of the pushout scheme $X$ is local, we may assume that all the schemes involved are  affine, namely $S=\Spec R$, $\bar X=\Spec \bar A$, $\bar Y=\Spec \bar B$, $Y=\Spec B $, $X=\Spec A$. We have a cartesian diagram
\[
\begin{CD}
A@>\psi>>\bar A \\
@V p VV @VV{\bar p}V\\
B @>>> \bar B
\end{CD}
\]
where $\bar p$ and $p$ are surjective and $\psi$ gives an isomorphism $I:=\ker p\to J:=\ker \bar p$ (cf.~Remark \ref{rem: cartesian}). Tensoring with $R/\mathfrak m_s$, where $\mathfrak m_s\subset R$ is the ideal of $s$, we obtain a commutative diagram:
\begin{equation}\label{eq: tensor}
\begin{CD}
A\otimes_R R/ \mathfrak m_s@>{\psi\otimes 1_R}>>\bar A\otimes_R R/\mathfrak m_s \\
@V {p \otimes 1_R}VV @VV{\bar p\otimes 1_R}V\\
B\otimes_R R/\mathfrak m_s @>>> \bar B \otimes_R R/\mathfrak m_s
\end{CD}
\end{equation}
where the vertical arrows are surjective. The claim is equivalent to diagram \eqref{eq: tensor} being cartesian. By Remark \ref{rem: cartesian}, this is in turn equivalent to the fact that $\psi$ gives an isomorphism $\ker (p \otimes 1_R)\to \ker (\bar p \otimes 1_R)$. 
Since $B$ is flat over $R$, we have an exact sequence: 
$$0=\Tor_1(B, R/\mathfrak m_s)\to I\otimes_R R/\mathfrak m_s\to A \otimes_R R/\mathfrak m_s\to B \otimes_R R/\mathfrak m_s\to 0,$$
so $\ker( p \otimes 1_R)=I\otimes_R R/\mathfrak m_s$. Analogously, we have $\ker (\bar p \otimes 1_R)=J\otimes_R R/\mathfrak m_s$ because $\bar B$ is also flat over $R$, and we conclude by noting that $\psi\otimes 1_R$ gives an isomorphism $I\otimes_R R/\mathfrak m_s\cong J\otimes_R R/\mathfrak m_s$.
\end{proof}

\section{Double covers}\label{sec: double}

In this section we consider quasi-projective varieties and we consider double covers in this category.

Throughout all the section we fix  quasi-projective varieties $Y$ and $\bar Y$ over the algebraically closed field $\K$ of characteristic $\ne 2$ and a finite flat degree $2$ morphism $g\colon \bar Y\to Y$. The sheaf $g_*\OO_{\bar Y}$ is locally free of rank 2 and there is a short exact sequence
$$0\to \OO_Y\to g_*\OO_{\bar Y}\to \mathcal Q\to 0.$$

\subsection{Definitions and notation} \label{ssec: double-def}

In this subsection we recall quickly the main facts about (flat) double covers and set the notation. More details can be found in  \cite{rita-abel} (for the normal case)   and \cite{rita-valery}. 
\medskip

Denote by $\Tr\colon g_*\OO_{\bar Y}\to\OO_Y$ the trace map:  then $\frac 12\Tr$ splits the above sequence and gives  a canonical decomposition $g_*\OO_{\bar Y}=\OO_Y\oplus \mathcal Q$. So the rank 1 sheaf  $\mathcal Q$ is projective, and therefore  free; it is traditional to write $\mathcal Q=L\inv$ for a suitable line bundle $L$. A local computation shows that if $\Tr(z)=0$ then $z^2\in \OO_Y$, so we can define an algebra involution  of $g_*\OO_{\bar Y}$ by defining it as the identity on $\OO_Y$ and multiplication by $-1$ on $L\inv$; we denote by $\iota$ the corresponding involution of $\bar Y$ and we remark  that $g\colon \bar Y\to Y$ is the corresponding quotient map. 

 The $\OO_Y$-algebra structure of $g_*\OO_{\bar Y}$ is determined by the multiplication map $s\colon L\inv \otimes L\inv \to \OO_Y$. The  surjection  $\sym(L\inv )\to \OO_Y\oplus L\inv$ induced by $s$ gives  a closed embedding $\bar Y\into V_Y(L\inv)$. In other words, over an open subset $U\subset Y$ such that $L|_U$ is trivial  $\bar Y$ is given by $\{z^2-b=0\}\subset U\times \A^1_z$, where $b$ is the local expression of $s$,  and $g$ is induced by the projection $U\times \A^1_z\to U$. The branch divisor  $D\subset Y$ of $g$ is the divisor of zeros of $s$ and the ramification locus $ R\subset \bar Y$ is  defined locally by $z=0$, so that $g^*D=2R$ and $\OO_{\bar Y}(R)\cong g^*L$. The cover is unbranched (\'etale) if $D=0$. 
 
 \begin{rem} If $Y$ is smooth and $\bar Y$ is $S_2$ (e.g., it is normal) then a  finite degree 2 morphism $\bar Y\to Y$  is automatically flat.
 \end{rem}
 \begin{rem}
  Two sections $s,s'\in H^0(Y, L^2)$ determine isomorphic double covers of $Y$ iff there is $\lambda\in H^0(Y, \OO_Y^*)$ such that $s'=\lambda^2 s$.  
  So if $H^0(Y,\OO_Y^*)=\K^*$ (e.g., $Y$ is projective) the double cover is determined up to isomorphism by the line bundle $L$ and by the choice of an effective divisor $D\in   |2L|$. Sometimes we say that $g$ is the double cover given by the relation $2L\sim D$. 
\end{rem}

\subsection{Embedding double covers in $\pp^1$-bundles}

Composing the natural  embedding $\bar Y\into V(L\inv)$ with the   inclusion $V(L\inv)\subset \pp(\OO_Y\oplus L\inv )$  one gets a closed embedding
 of $\bar Y$  as a bisection of $\pp(\OO_Y\oplus L\inv)$. In this subsection we show that  $\bar Y$ has a closed embedding as a  bisection of 
 the $\pp^1$-bundle $\pp_Y(g_*M)$ for any line bundle $M$ on 
 $\bar Y$.
\medskip

Fix a line bundle $M\in \Pic(\bar Y)$ and set $E:=g_*M$. 
We denote by $p\colon \pp_Y(E)\to Y$ the projection and by $h$ the class of $\OO_{\pp_Y(E)}(1)$ in the Picard group. 
Then one has:
\begin{prop} \label{prop: embedding}
In the above setup: 
\begin{enumerate}
\item 
The natural map $g^*E\to M$ is a surjection that induces a  morphism  $\tilde g\colon \bar Y\to \pp_Y(E)$ 
\item $\tilde g$ is a closed  embedding and $\tilde g(\bar Y)$ is a Cartier divisor 
\item   $\tilde g(\bar Y)$ is linearly equivalent to  $2h+p^*(L-\det E)$.
\end{enumerate}
\end{prop}

Before  proving   Proposition \ref{prop: embedding}  we  note the following  elementary fact. 

\begin{lem}\label{lem: easy} 
Let $M$ be a line bundle on $\bar Y$. Then there is an affine open cover $\{U_i\}_{i\in I}$  of $Y$ such that $M|_{g\inv U_i}$ is trivial for all $i\in I$.
\end{lem}
\begin{proof}
Pick  $y\in \bar Y$ and choose a very ample effective divisor $D$  such that $y, \iota (y)\notin D$.  For $d\gg 0$ the line bundle  $M':=M(dD)$ is globally generated. In particular, $M'$ has a section that is  non-zero at  the points $y$ and $\iota (y)$;   since $M'$ and $M$ are isomorphic on $\bar Y\setminus D$, it follows that $M$ can be trivialized on some affine open subset $V$ containing $y$ and $\iota (y)$. Setting $U:=g(V\cap \iota V)$  we have $g\inv (U)=V\cap \iota V$ and $M|_{g\inv U}$ is trivial. 
\end{proof}

\begin{proof}[Proof of Prop. \ref{prop: embedding}]
(i) Let $y\in Y$ be a point. Since $g$ is flat and finite, by cohomology and base change we have a canonical isomorphism  $E\otimes_{\OO_Y}\K(y)\cong H^0(g\inv(y), M_{|g\inv(y)})$. So, given $x\in g\inv(y)$ the map $g^*E\otimes_{\OO_{\bar Y}}\K(x)\to M\otimes_{\OO_{\bar Y}}\K(x)$ coincides with the restriction map $H^0(g\inv(y), M_{|g\inv(y)})\to H^0(\{x\}, M_{|\{x\}})$ and is therefore  surjective.  It follows that $g^*E\to M$ is surjective.
\medskip

(ii) The claim is local on $Y$. By Lemma \ref{lem: easy}, up to replacing  $Y$ by a suitable open subset we may assume that $M$ is the trivial bundle. So $\tilde g$ coincides with the usual closed embedding of $\bar Y$ as a Cartier divisor of $\pp_Y(\OO_Y\oplus L\inv)$. 
\medskip
 
(iii) Clearly, it is enough to prove the claim  for each connected  component of $Y$, hence we may assume  that $Y$ is connected.   The  Picard group of $\pp_Y(E)$ is generated by $h$ and $p^*(\Pic(Y))$;  since  $\tilde g(\bar Y)$ is a Cartier divisor that is a bisection of $p$,  we may write $\tilde g(\bar Y)=2h+p^*(\Delta)$ for some $\Delta \in \Pic(Y)$. 

Consider the restriction sequence
$$0\to \OO_{\pp_Y(E)}(-2h-p^*\Delta)\to  \OO_{\pp_Y(E)}\to \OO_{\tilde g(\bar Y)}\to 0;$$
pushing forward to $Y$ we obtain the exact sequence
\begin{equation}\label{eq: push}
0\to \OO_Y\to g_*\OO_{\bar Y}\to R^1p_* \OO_{\pp_Y(E)}(-2h-p^*\Delta)\to 0.
\end{equation}
We have  $$R^1p_* \OO_{\pp_Y(E)}(-2h-p^*\Delta)\cong R^1p_* \OO_{\pp_Y(E)}(-2h)\otimes \OO_Y(-\Delta)\cong (\det E)\inv \otimes \OO_Y(-\Delta),$$ where the first isomorphism is given by the projection formula and  the second one by \cite[Ex.~III.8.4]{Hartshorne}. Since $g_*\OO_{\bar Y}=\OO_Y\oplus L\inv$, the claim follows by taking determinants in \eqref{eq: push}. 
\end{proof}

\subsection{An explicit gluing construction}\label{ssec: glue-bundle}

We keep the notation of the previous section. 
\medskip

Set $\bar X:= V_{\bar Y}(M)$ and embed $\bar Y$ into $\bar X$ as the zero section; since quasi-projective varieties satisfy condition (AF) (cf.~Remark \ref{rem: AF}), by Theorem \ref{thm: glue} the pushout scheme  $X:=\bar X\sqcup_{\bar Y} Y$ exists. 
 In this section we prove the following result, which may be seen as a global version of Example \ref{ex: pinch}:

 \begin{thm}\label{prop: normal-bundle}
 Denote by  $q\colon V_Y(E)\to Y$  the natural projection; then 
there is a closed embedding $X\into V_{Y}(E)$ such that
  \[
  \OO_{V_Y(E)}(X)=q^*(L \otimes ( \det E)\inv).
    \]
    \end{thm}
 We note the following consequence of the above theorem, which is useful in computations:
  \begin{cor}\label{cor: utile}
 $$g^*\OO_{V_Y(E)}(X)=\OO_Y(g^*D)\otimes M\inv\otimes \iota^*M\inv.$$
 \end{cor}
 \begin{proof}
Theorem  \ref{prop: normal-bundle} gives $g^*\OO_{V_Y(E)}(X)=g^*(L \otimes ( \det E)\inv)$, so the claim follows by Lemma \ref{lem: determinant} below, recalling that $D$ is linearly equivalent to $2L$.
 \end{proof}
 \begin{lem}\label{lem: determinant}
 One has:
 $$g^*(\det E)=M\otimes\iota^*M\otimes g^*L\inv.$$
 \end{lem}
 \begin{proof}
 Since $g\circ \iota=g$ and $\iota^2=1$, there is natural isomorphism $$E=g_*M\cong g_*(\iota_*M)=g_*(\iota^*M).$$
So we have a short sequence
 \begin{equation}\label{eq: det}
 0\to g^*E\overset{\alpha}{\to} M\oplus \iota^*M\overset{\beta}{\to}  M|_R\to 0
 \end{equation}
 where $R$ is the ramification divisor of $g$,  the components of $\alpha$ are  the natural maps $E\to M$ and $E\to\iota^*M$ and $\beta(x, y)=x|_R-y|_R$.  A local computation (cf.~Lemma \ref{lem: easy}) shows that \eqref{eq: det} is actually exact. Since $\OO_{\bar Y}(R)=g^* L$ (see \S \ref{ssec: double-def}), we have $g^*(\det E)=\det g^*(E)=M\otimes\iota^*M\otimes\det (M|_R)\inv = M\otimes\iota^*M\otimes g^*L\inv$.
  \end{proof}
\smallskip

 \begin{proof}[Proof of Thm. \ref{prop: normal-bundle}]
 Let $\pi\colon V_Y(E)\setminus Y\to \pp_Y(E)$  be the projection. Let $\tilde g\colon \bar Y\to \pp_Y(E)$ be the closed embedding defined in Proposition \ref{prop: embedding} and let  $Z\subset V_Y(E)$ be the closure of $\pi^*(\tilde g(\bar Y))$ (in other words, $Z$ is the relative affine cone over $\tilde g(\bar Y)\subset \pp(E)$). 
 By Proposition  \ref{prop: embedding}, (iii),  the hypersurface  $Z$  is  in $|q^*(L-\det E)|$. 

 We wish to show that there is an isomorphism $X\to Z$. As a first step we define a map $\Psi\colon X\to Z$ using the universal property of pushout schemes, as follows.
 
 The inclusion $i_Y\colon Y\to V_Y(E)$ induces  a closed embedding $j\colon Y\to Z$. The surjection $g^*E\to M$ (cf.~Proposition \ref{prop: embedding} (i)) induces a morphism  $\Phi\colon \bar X=V_{\bar Y}(M)\to V_Y(E)$ whose restriction to the fibers of $V_{\bar Y}(M)$ is linear and injective; restricting $\Phi$  to $\bar X\setminus \bar Y\to  V_Y(E)\setminus Y$ we obtain a commutative diagram
 \[\begin{CD}
\bar X \setminus \bar Y @>{\Phi|_{\bar X\setminus \bar Y}}>>V_Y(E)\setminus Y\\
@V VV @VqVV\\
\bar Y@>>{\wt g}>\pp_Y(E)
\end{CD}
\]
that shows that $\Phi( \bar X)=Z$.  Since $\Phi|_{\bar Y}=j\circ g$, by the universal property of $X$  there is a unique map $\Psi\colon X\to Z$  induced by $(\Phi,i_Y)$. 
Since  the underlying set of $X$ is  the pushout in the category of sets   (cf.~\cite[Scolie 4.3]{Ferrand}),  it is immediate to check that $\Psi$ is a bijection.  

We prove that $\Psi$ is an isomorphism by using a local argument.
 
By Lemma \ref{lem: easy}  we  may replace  $Y$ by an affine open subset $U$ such that $M|_U$ and $L|_U$ are trivial,  $\bar Y$ by $\bar U:=g\inv(U)$, and $\bar X$ by $\bar U\times \A^1$. We have $U=\Spec B$, $\bar U=\Spec \bar B$,  where  $\bar B=B\oplus Bz$ is a free $B$-module of rank two and  $z^2=b\in B$.
The map $\Phi$ defined above restricts to the map $\bar U\times \A^1_t\to U\times \A^2_{u,v}$ defined by $(x, t)\mapsto (g(x),zt, t)$ with image $W:=\{u^2-bv^2=0\}\subset  U\times \A^2_{u,v}$. The affine variety $T$ obtained by pinching  $\bar U\times \A^1$  along $\bar U\times \{0\}$ via $g|_{\bar U}$ is an open subvariety of $X$; 
to prove  that $\Phi$ induces an isomorphism $V\to W$, it suffices to show that the following diagram is cocartesian: 

 \[\begin{CD}
\bar U@>>>U \\
@V VV @VVV\\
\bar U\times \A^1 @>>>W
\end{CD}
\]
In turn, this is the same as showing that the dual diagram 
 \begin{equation}\label{eq: diag2}
 \begin{CD}
B[u,v]/(u^2-bv^2)@>\psi>>\bar B[t] \\
@V p VV @VV{\bar p}V\\
B @>>> \bar B
\end{CD}
\end{equation}
is cartesian. The maps $p$ and $\bar p$ are surjective, so by \cite[Lemme 1.2]{Ferrand} diagram \eqref{eq: diag2} is cartesian iff $\psi$ induces an isomorphism   $\ker p\to \ker \bar p$. It is easy to see that $\psi$ is injective. Since $\ker \bar p$ is generated  as a $B$-module by the monomials $t^i$ and $zt^i$ with $i\ge 1$, the map $\ker p\to \ker \bar p$ is also surjective.
\end{proof}\label{prop: facile}

Assume now that  $Y$ and $\bar Y$ are smooth.  By Theorem \ref{prop: semi-smooth-glue} the pushout scheme $X$ is a semi-smooth variety and the set of pinch points of $X$ coincides with the branch locus $D\subset Y$ of the flat double cover $g\colon \bar Y\to Y$. We have seen (Lemma \ref{lem: IXsing}) that for  a semi-smooth variety the ideal $I_{Y|X\sing}$ is an invertible sheaf on $D$.  In the situation we are considering   it is possible  to determine $I_{Y|X\sing}$  explicitly:
\begin{prop} \label{prop: ideal} In the above setup, assume that $Y$ and $\bar Y$ are smooth and denote by $r\colon R\to D$ the isomorphism induced by $g$. Then:
$$I_{Y|X\sing}\cong r_*(M|_R).$$

\end{prop}
\begin{proof} The first step of the Koszul resolution for the ideal   of $I_Y$ of $Y$ in $V_Y(E)$ is a surjection $q^*E\onto I_Y$. 
Since  $I_{Y|X\sing}$ is supported on $D$, restriction to $D$ gives  a surjection $\psi\colon E|_D\onto I_{Y|X\sing}$.  
On the other hand, restricting \eqref{eq: det} to $R$ and pushing down to $D$ we obtain an exact sequence:
\[%\begin{equation}
E|_D\to r_*(M|_R)\oplus r_*(M|_R)\overset{\beta}{\to}  r_*(M|_R)\to 0
\]%\end{equation}
where $\beta(s_1,s_2)=s_1-s_2$. So we have a surjection $\phi\colon E|_D\onto \ker \beta=r_*(M|_R)$. 
Since $I_{Y|X\sing}$  is an invertible sheaf on $D$ by Lemma \ref{lem: IXsing}, to prove our claim it is enough to show that $\ker \phi\subseteq \ker \psi$.

We show this inclusion by means a local computation, arguing as in the last part of the proof of Theorem \ref{prop: normal-bundle} and using the same notation. We work in a neighbourhood $U=\Spec B$ of a point $p\in D\subset Y$, so that  $u(p)=v(p)=b(p)=0$ and $b$ is a coordinate on $U$. We  assume that $M$ is trivial on $\bar U:=g\inv(U)=\Spec \bar B$, where $\bar B=B\oplus Bz$, with $z$ an antiinvariant function such that $z^2=b$.  If we take as $e_1:=z$ and $e_2:=1$ as a local basis of $E$ and 1 as a local generator for $M$, then  the map $g^*E\to M\oplus \iota_*M$  is given locally by $e_1\mapsto (z,-z)$ and $e_2\mapsto (1,1)$, so the kernel of $\phi$ is spanned by $e_1$. 

We let $u,v$ be the coordinates on $q^*E$ dual to the local basis $e_1,e_2$:  on  $U\times \A^2_{u,v}$ the map $q^*E\onto I_Y\subset B[u,v]$ can be written locally as $ B[u,v]e_1 \oplus B[u,v]e_2  \overset{(u,v)}{\to} B[u,v] $. As in the proof  of Theorem \ref{prop: normal-bundle}   the pushout scheme $X$ is defined inside $U\times \A^2_{u,v}$
 and  $X\sing =\Spec B[v, u]/(u, v^2, bv)$, so the above map, when restricted to $X\sing$, sends $e_1$ to zero. A fortiori $e_1$ is in the kernel of $\psi$, as required. 
\end{proof}

\section{Computing $T_X$ and  $\mathcal T^1_X$ of a  semi-smooth variety}\label{sec: def}

In this section  $X$  is a semi-smooth variety over $\K$ (cf.~\S \ref{sec: semi-smooth}).  We use freely the notation of \S \ref{sec: semi-smooth} and \S \ref{ssec: double-def}; given a
  sheaf $\mathcal F$  on $\bar Y$ and a linearization of $\mathcal F$ with respect to $\iota$, we denote by  $(g_*\mathcal F)^{\rm inv}$ the invariant subsheaf of $g_*\mathcal F$.

\subsection{The tangent sheaf of a semi-smooth variety}\label{ssec: tg}

Here we describe the tangent sheaf $T_X$ in terms of the normalization map $f\colon \bar X\to X$. Our results are summarized in the following:

\begin{thm} \label{prop: tg}
Let $X$ be a semi-smooth variety,  let $f\colon \bar X\to X$ be the normalization map, let $\bar Y \subset \bar X$ and 
$Y\subset X$ the subschemes defined by the conductor  (cf. \S \ref{ssec: A2}) and let $g\colon \bar Y\to Y$ the degree 2 map induced by $f$. Then:

\begin{enumerate}
\item there is a natural injective map $\alpha\colon  T_X\to f_*T_{\bar X}$ which is an isomorphism on  the smooth locus of $X$;

\item set $\mathcal G:=\coker \alpha$;  then $\alpha$ induces   an exact sequence
$$0\to (g_*T_{\bar  Y})^{\rm inv}\to g_*T_{\bar  X}|_{\bar  Y}\to \mathcal G\to 0.$$
\end{enumerate}
\end{thm}
\medskip

The rest of the section is devoted to proving Theorem  \ref{prop: tg}. To simplify the notation, we write down  the proof in the two-dimensional case; the arguments in the higher dimensional case are exactly the same. 

Dualizing the natural map $f^*\Omega_X\to \Omega_{\bar X}$ we obtain  a natural injective map $j\colon T_{\bar  X}\to (f^*\Omega_X)^{\vee}$.
\begin{lem}\label{lem: iso-tg}
The map  $j\colon T_{\bar X}\to (f^*\Omega_X)^{\vee}$ is an isomorphism. 
\end{lem}
\begin{proof}
Set $\mathcal F:= (f^*\Omega_X)^{\vee}$.   If  locally in the \'etale   topology  $X$ is given by $\{h(u,v,w)=0\}\subset \A^3_{u,v,w}$, the exact  sequence of differentials \begin{equation}\label{seq: differentials}
0\to \OO_X(-X)=\OO_X\overset{^t(\frac{\partial h }{\partial u}, \frac{\partial h }{\partial v}, \frac{\partial h }{\partial w})}{\longrightarrow}\Omega_{\A^3|X}=\OO_X^{\oplus 3}\to \Omega_X\to 0
\end{equation} 
is exact. Pulling back to $\bar X$ and dualizing we obtain the following exact sequence on $\bar X$: 
\begin{equation}\label{eq: pres-F}
0\to \mathcal F\to \OO_{\bar  X}^{\oplus 3}\overset{(\frac{\partial h}{\partial u}, \frac{\partial h }{\partial v}, \frac{\partial h }{\partial w})}{\longrightarrow}  \OO_{\bar  X}.
\end{equation}
that  shows that $\mathcal F$ is $S_2$.

Denote by $R\subset \bar  X$ the preimage of the set of the pinch points  of $X$, which is a codimension 2 closed subset by assumption, and set $U:=\bar X\setminus R$. 
We are going to show that   $j$ restricts to an isomorphism on $U$. This will finish the proof, since a map of $S_2$ sheaves that is an isomorphism in codimension 1 is an isomorphism.

 Locally  in the \'etale  topology near a double crossings point of $X$ we may assume    $\bar X=\{z^2-1=0\}\subset \A^3_{x,y,z}$, $X=\{uv=0\}\subset \A^3_{u,v,w}$ and $(x,y, t)\overset{f}{\mapsto} ((z-1)y,(z+1)y, x)$. 
The map $\OO_{\bar  X}^{\oplus 3}\to  \OO_{\bar  X}$ of \eqref{eq: pres-F} is given by $((z+1)y, (z-1)y, 0)$, hence  $\mathcal F$ is locally generated by $(0,0,1)$ and $(z-1, z+1,0)$. Finally, $j$ maps $\frac{\partial }{\partial y}$ to $(z-1, z+1,0)$ and  $\frac{\partial }{\partial x}$ to $(0, 0,1)$.
\end{proof} 
\begin{rem} The proof of Lemma \ref{lem: iso-tg} works more generally for $X$ locally  hypersurface  and demi-normal. 
\end{rem}

\begin{proof}[Proof of Thm.  \ref{prop: tg}]
(i) Consider  the  natural map 
$$f^*T_X=f^*(\homc_{\OO_X}(\Omega_X, \OO_X))\to \homc_{\OO_{\bar  X}}(f^*\Omega_X, \OO_{\bar  X})=(f^*\Omega_X)^{\vee},$$ which is an isomorphism on $\bar X\setminus \bar Y$. 
Composing this map with  the isomorphism $j\inv \colon f^*(\Omega_X)^{\vee}\to T_{\bar X}$ (cf.~Lemma  \ref{lem: iso-tg}) we get a map $f^*T_X\to T_{\bar  X}$.  Pushing down to $X$ and composing with the natural map $T_X\to f_*(f^*T_X)$ gives   the  map $\alpha \colon T_X\to f_*T_{\bar  X}$, which is  an isomorphism  on $X\setminus Y$, and therefore is injective, since $T_X$ is torsion free. So  we have an exact sequence:
\begin{equation}\label{eq: tg-sequence}
0\to T_X\to f_*T_{\bar  X}\to \mathcal G \to 0,
\end{equation}
where $\mathcal G$ is supported on $Y$. 
\medskip

(ii) Follows from Lemma \ref{lem: G} below. 
\end{proof}
\begin{lem}\label{lem: G}
\begin{enumerate}
\item  The morphism $f_*T_{\bar  X}\to \mathcal G$ factors via $g_*T_{\bar  X}|_{\bar  Y}$;
\item The kernel of the induced map   $g_*T_{\bar  X}|_{\bar  Y}\to \mathcal G$ is $g_*(T_{\bar  Y})^{\text {inv}}$, the invariant part of $g_*(T_{\bar  Y})$.
\end{enumerate}
\end{lem}
\begin{proof} 
Since the map $f$ is finite, we have a short exact sequence:
$$0\to f_*T_{\bar X}(-\bar Y)\to f_*T_{\bar X}\to f_*T_{\bar X}|_{\bar Y}=  g_*T_{\bar X}|_{\bar Y}\to 0, $$
so to prove (i)  it is enough to show that the composition $f_*T_{\bar X}(-\bar Y)\to f_*T_{\bar X}\to \cG$ is 0;  then to prove (ii)  one needs to show that the sequence $0\to  g_*(T_{\bar  Y})^{\text {inv}}\to g_*T_{\bar X}|_{\bar Y}\to \cG\to 0$ is exact. Since $X$ is  semi-smooth,  it is enough to prove both statements in the situation of Example  \ref{ex: pinch}.

It is enough to consider the case 
   $\bar X=\A^2_{x,y}$, $X=\{u^2-v^2w=0\}\subset \A^3_{u,v,w}$ and $f\colon \bar X\to X$ defined by $(x,y)\mapsto(xy,y,x^2)$. So we have $Y=\{u=v=0\}$, $\bar Y=\{y=0\}$   and the tangent sheaf $T_X$ is the kernel of 
$T_{\A^3}|_{X}=\OO_X^{\oplus 3}\xlongrightarrow{(2u,-2vw, -v^2)}\OO_X=\OO_X(X)$.

  A set of generators of $T_X$ is given by:

\begin{gather*}
e_1:=vw\frac{\partial  }{\partial u}+u \frac{\partial  }{\partial v}, \quad  e_2:=u\frac{\partial  }{\partial u}+v \frac{\partial  }{\partial v},\\ e_3:= v^2 \frac{\partial  }{\partial u}+2u\frac{\partial}{\partial w}, \quad e_4:= v\frac{\partial  }{\partial v}-2w \frac{\partial  }{\partial w}.
\end{gather*} 
Since $f_*\OO_{\bar X}$ is generated by $1,x$ as an $\OO_X$-module, the sheaf $f_*T_{\bar X}$ is generated as an $\OO_X$-module by
 $$\frac{\partial  }{\partial x}, \frac{\partial  }{\partial y}, x\frac{\partial  }{\partial x}, x\frac{\partial  }{\partial y}.$$
The chain rule gives relations: 
$$\frac{\partial  }{\partial x}=y\frac{\partial  }{\partial u}+2x\frac{\partial  }{\partial w},\quad  \frac{\partial  }{\partial y}= x\frac{\partial  }{\partial u}+\frac{\partial  }{\partial v}.$$
Therefore we have:
\begin{gather*} \alpha(e_1)=xy\frac{\partial  }{\partial y},\quad \alpha(e_2)=y\frac{\partial  }{\partial y}, \\
\alpha(e_3)=y\frac{\partial  }{\partial x}, \quad \alpha(e_4)= y\frac{\partial  }{\partial y}-x\frac{\partial  }{\partial x}
\end{gather*}
and $\alpha(T_X)$ is the subsheaf generated by $u\frac{\partial  }{\partial y}, v\frac{\partial  }{\partial y}, v\frac{\partial  }{\partial x}, x\frac{\partial  }{\partial x}$.
The sheaf  $f_*T_{\bar X}(-\bar Y)$ is generated by  $  u\frac{\partial  }{\partial y}, v\frac{\partial  }{\partial y}, v\frac{\partial  }{\partial x}, u\frac{\partial  }{\partial x}=vx\frac{\partial  }{\partial x}$, so we see that $f_*T_X(-Y)\subset \alpha(T_X)$ and that the quotient sheaf $\alpha(T_X)/f_*T_X(-Y)$ is generated by $x\frac{\partial  }{\partial x}$, i.e.,  (i) and (ii) hold in this  case. 

\end{proof}

\subsection{The sheaf $\mathcal T^1_X$ for a semi-smooth variety}\label{ssec: T1}

As in  the previous section we assume that $X$ is  a semi-smooth variety with reduced singular locus $Y$, $f\colon\bar X\to X$ is  the normalization,  $\bar Y \subset \bar X$ and $Y\subset X$ are the subschemes defined by the conductor  (cf. \S \ref{ssec: A2}).  %We let  $g\colon \bar Y\to Y$ be the degree 2 map induced by $f$, we let  $\iota\colon \bar Y\to \bar Y$ the corresponding involution and as usual we denote by $D\subset Y$  the branch locus of $g$ and by $L\inv$ the antinvariant subhaf of $g_*\OO_

\medskip
Observe   (see Def. \ref{def: SIng})  that    there is an exact sequence:
\begin{equation}\label{eq: T1}
0\to \mathcal T^1_X\otimes \mathcal I_{Y|X\sing}\to \mathcal T^1_X\to \mathcal T^1_X|_Y\to 0. 
\end{equation}

Let $N_{\bar Y|\bar X}=\OO_{\bar Y}(\bar Y)$ be  the normal bundle of $\bar Y$ in $\bar X$; the main result of this section is the explicit computation of the first and last term in \eqref{eq: T1}:

\begin{thm} \label{thm: computing-T1}
In the above set up   let  $D$ be  the branch locus of $g$ and let $L\inv$ be  the anti-invariant summand of $g_{\ast} \OO_{\bar Y} $.

Then we have the following isomorphisms of line bundles: 

\begin{enumerate}
 \item  on $Y$, $\cT^1_X|_Y\cong L\otimes (\det g_*(N_{\bar Y|\bar X}\inv ))\inv$;
\item on $\bar Y$, $g^*(\cT^1_X|_Y)\cong   g^*(L^{\otimes 2})\otimes N_{\bar Y|\bar X}\otimes \iota^*N_{\bar Y|\bar X}$;
\item  on  $D$, $ I_{Y|X\sing}\cong  r_* \left( (N_{\bar Y|\bar X}\inv)_{|R} \right) $.
\end{enumerate}
\end{thm}

 This is one of the key technical points of this paper: its proof combines the results of \S  \ref{ssec: glue-ext} and \S \ref{ssec: glue-bundle} with the degeneration to the normal cone (cf.~Proposition \ref{prop: degeneration}). 
\bigskip

Let $\bar X_0$ be the total space of the normal bundle of $\bar Y$ in $\bar X$; let $X_0$ be the semi-smooth variety obtained by pinching $\bar X_0$ along $\bar Y$ via $g\colon \bar Y\to Y$.

\begin{prop}\label{prop: degeneration} We can construct a degeneration  of $X$ to $X_0$, i.e. a cartesian diagram 
\begin{equation}\label{eq: diag3}
\begin{CD}
Y @>>> Y\times \mathbb A^1 @<<< Y\times U\\
@VVV @VVV @VVV\\
X_0 @>>>\mathcal X @<<< X\times U\\
@V{q_0}VV @VqVV @VV{q_U}V\\
{0}@>>>\mathbb A^1 @<<< U
\end{CD}
\end{equation}
where $U:=\mathbb A^1\setminus\{0\}$, $\mathcal X$ is a semi-smooth variety,  $Y\times \mathbb A^1$ is its reduced singular locus, and 
the morphism $q\colon \mathcal X\to \mathbb A^1$ is flat and lci. 
\end{prop}
\begin{proof}
Let $\bar{\mathcal X}$ be the degeneration to the normal cone (in fact, bundle) of the embedding $\bar Y \to \bar X$; it is a nonsingular variety with a smooth morphism $\bar q\colon\bar{\mathcal X}\to \mathbb A^1$ and a natural closed embedding
 of $\bar{\mathcal Y}:=\bar Y\times \mathbb A^1$  "into  $\bar {\mathcal X}$   such that the normal bundle $N_{\bar{\mathcal Y}/\bar{\mathcal X}}$ is the pullback from $\bar Y$ of $N_{\bar Y/\bar X}$. This can be easily checked by hand, since $\bar{\mathcal X}$ is the blow up of $\bar X\times \mathbb A^1$ along $\bar Y\times \{0\}$ minus the strict transform of $\bar X\times \{0\}$, which is isomorphic to $\bar X$; details can be checked in \cite[\S 5.1]{Fulton}, in particular Example 5.1.2.

We define $\mathcal X$ to be the variety obtained by pinching $\bar{\mathcal X}$ along $\bar Y\times \mathbb A^1$ via
 $(g,\Id_{\mathbb A^1})\colon \bar Y\times\mathbb A^1\to Y\times \mathbb A^1$. By Prop \ref{prop: semi-smooth-glue}, the variety $\bar{\mathcal X}$ is semi-smooth, and the
  conductor ideal defines a closed embedding $Y\times \mathbb A^1\to \mathcal X$.
 The map $q\colon \mathcal X\to \mathbb A^1$ is induced by the projection $\bar{\mathcal X}\to \bar X\times \mathbb A^1\to \mathbb A^1$ and by 
 the pushout property; it is flat since every component of $\mathcal X$ dominates $\mathbb A^1$.

The cartesian diagrams on the right hand side of the diagram \ref{eq: diag3}  follow immediately because the pushout construction commutes with product with $U$. The cartesian diagram on the left follows from Prop \ref{prop: special_glue}.

Away from $Y\times \mathbb A^1$, $q$ is smooth because $\bar q$ is. The morphism $\bar q\colon\bar{\mathcal X}\to \mathbb A^1$ is smooth, and so are the projections of $Y\times \mathbb A^1$ and $\bar Y\times \mathbb A^1$ to $\mathbb A^1$. Hence in the argument in \S \ref{ssec: A1} we can always assume that one of the local coordinates is the parameter $t$ of $\mathbb A^1$; thus, \'etale locally, $\mathcal X$ is the product of a semi-smooth variety 
with $\mathbb A^1$ and $q$ is the projection, which is locally hypersurface, hence lci.

\end{proof}

\begin{cor}\label{cor: same bundle} There is an isomorphism of invertible sheaves on $Y$ between $\mathcal T^1_{X_0}|_Y$ and $\mathcal T^1_X|_Y$.
\end{cor}
\begin{proof} By Theorem \ref{thm: restriction} we have canonical isomorphisms $\mathcal T^1_q|_{X_0}\cong \mathcal T^1_{X_0}$ and, for every $t\ne 0$, $\mathcal T^1_q|_{\mathcal X_t}\cong \mathcal T^1_{\mathcal X_t}$. Let $\cL_q$ be the restriction of $\mathcal T^1_q$ to $Y\times \mathbb A^1$; it is a coherent sheaf whose restriction to each fiber $Y\times \{t\}$ is an invertible sheaf, hence it is itself an invertible sheaf. Thus it defines a morphism from $\mathbb A^1$ to the Picard variety of $Y$, which is constant on $U$ since all fibers $\mathcal X_t$ for $t\ne 0$ are isomorphic to $X$. Since the Picard variety is separated, this morphism is constant.
\end{proof}

As in \S \ref{ssec: double-def}, we write  $g_*\OO_{\bar Y}=\OO_Y\oplus L\inv$ and we denote by $D$ the branch locus of $g$  (so, in particular, $2L\sim D$); 
recall (Remark \ref{rem: semi-smooth-glue}) that $D$ is the subset of pinch points of $X$.  We set $E:=g_*(N_{\bar Y|\bar X}\inv)$.

 \begin{proof}[Proof of Thm. \ref{thm: computing-T1}]
(i) By Prop \ref{prop: degeneration} we can construct a degeneration of $X$ to the semi-smooth variety $X_0$ obtained by pinching the total space of the normal bundle of $\bar Y$ in $\bar X$ along the zero section via the map $g\colon \bar Y\to Y$. By Corollary \ref{cor: same bundle} we have that $\mathcal T^1_X|_Y$ is isomorphic to  $\mathcal T^1_{X_0}|_Y$. 

By Theorem \ref{prop: normal-bundle}, we can construct a closed embedding of $X_0$  as a hypersurface in the total space $V_Y(E)=\Spec\sym(E)$. By Remark \ref{rem: T1-divisor}, we have that $\mathcal T^1_X|_Y$ is isomorphic to $\mathcal O_{V_Y(E)}(X_0)|_Y$, which again by Theorem \ref{prop: normal-bundle} gives the result.

(ii) follows immediately from (i) and Corollary \ref{cor: utile}.

(iii) By Corollary \ref{cor: sing_pullback} we have that $\mathcal X\sing\cap X_0=(X_0)\sing$, while $\mathcal X\sing\cap (X\times U)=X\sing\times U$
 follows from the definition. Since $\mathcal X$ is semi-smooth, it follows from Lemma \ref{lem: IXsing} that $I_{Y\times \mathbb A^1|\mathcal X}$ is a line bundle on $D\times \mathbb A^1$;
 for every $t\in \mathbb A^1$, its restriction to $D\times \{t\}$ surjects to $I_{Y|(\mathcal X_t) \sing}$; 
 since both are line bundles, the restriction is an isomorphism. It follows, by separatedness of the Picard variety of $D$, that $I_{Y|X\sing}$ is isomorphic
  to $I_{Y|(X_0){\rm sing}}$ and we conclude by  Proposition \ref{prop: facile}.

 \end{proof}

  \appendix
  \section{Proof of Proposition \ref{prop: semi-smooth-glue}}\label{app: A}
  We use freely the notation of \S \ref{sec: glue}.
   \subsection{Proof of $(ii)\Rightarrow (i)$}\label{ssec: A1}
   As usual we denote by $f\colon \bar X\to X$ the gluing map and by $\iota$ the involution of $\bar Y$ associated with $g\colon \bar Y\to Y$. \par
  
  Given a point $P\in  X$ we are going to show that $P$ has an affine neighbourhood $U_P$ with an \'etale map  $\phi_P\colon U_P\to P_n=\Spec\K[u,v,w]/(u^2-v^2w)\times\A^{n-2}$.  To define $\phi_P$ we consider the preimage $\bar U_P\subset \bar X$ of $U_P$ and define  
    an \'etale   map $\bar\phi_P\colon \bar  U_P\to\A^n_{x,y, t_1,\dots t_{n-2}}$ such that  $\bar Y_P:=\bar Y\cap \bar U_P$ is mapped to ${\A^{n-1}_{x,t_1,\dots t_{n-2}}}$ and there is a commutative diagram:
    \begin{equation}\label{eq: double-diag}
 \begin{CD}
Y_P @<g<< {\bar Y_P} @>>>{\bar U_P}\\
@VVV @VVV @VV{\bar \phi_P}V\\
\A^{n-1}_{w,t_1,\dots t_{n-2}} @<h<< {\A^{n-1}_{x,t_1,\dots t_{n-2}}} @>>> {\A^n_{x,y,t_1,\dots t_{n-2}}}
\end{CD}
   \end{equation}
      where $Y_P:=Y\cap U_P$, the horizontal arrows to the right are inclusions and $h(x,t_1,\dots t_{n-2})=(x^2 , t_1,\dots t_{n-2})$. By the universal property of  pushout schemes there is an induced map $\phi_P\colon U_P\to P_n$. Finally, by \cite[Lem.~ 44]{kollar-quotients} the map $\phi_P$ is \'etale if in diagram \eqref{eq: double-diag} the vertical maps are \'etale and both squares are cartesian.
      \bigskip
      
  We may of course assume that $P\in Y$; we have two cases according to whether $(a)$  $g\inv(P)$ is a single point $Q$,  or $(b)$  $g\inv(P)$ consists of two points $Q_1,Q_2$. Case $(a)$ will give a pinch point and case $(b)$ a double crossings point.
  
Since our arguments often involve passing to smaller affine neighbourhoods, we find it useful to note the following elementary result:
\begin{lem}\label{lem: open}
In the above setup,  if  $V\subset \bar X$ is an open subset containing $f\inv(P)$, then there exists an open affine  neighbourhood $U_P$ of $P\in X$ such that $f\inv(U_P)\subseteq V$.
\end{lem}  
\begin{proof} Set $Z:=\bar X\setminus V$; then $X\setminus f(Z)$ is an open neighbourhood of $P$, so it contains an affine open neighbourhood  $U_P$  of $P$ and we have $f\inv(U_P)\subseteq V$ by construction. 
\end{proof}
  Since the question is  local on $X$   we may assume  the following (cf.~Remark \ref{rem: glue-local}): 
 \begin{itemize}
\item[(1)] $\bar X$ is affine and equal to $\Spec \bar A$;
\item[(2)]  $\bar Y=\Spec \bar B$ where $\bar B=\bar A/I$ for  $I$  a principal ideal with generator $y$;
\item[(3)]  $Y=\Spec B$ and there is an anti-invariant element $\bar x\in \bar B$ such that $\bar B= B\oplus \bar x B$.
\end{itemize}
To see why condition (2) holds, first note that
 in case $(a)$ the  ideal $I$ is  principal in a  neighbourhood of $P$, since by assumption the variety $\bar X$ is smooth   and $\bar Y$ is  a divisor. 
  In case $(b)$ we may find $y_1,y_2\in I$ such that $y_i$ has nonzero differential at $Q_i$, $i=1,2$. So at least one among $y_1$, $y_2$ and $y_1+y_2$
   has nonzero differential at both $Q_1$ and $Q_2$ and therefore generates $I$ in an open set containing $Q_1$ and $Q_2$.
    In both cases by Lemma \ref{lem: open} we can then  shrink $X$ in such a way that condition (2) holds. 

We can assume that condition (3)  holds by the discussion of  section \ref{ssec: double-def}. 
 \smallskip

Consider case $(a)$ first:  the function $\bar x$ defines  the ramification divisor $R\subset \bar Y $ of  the double cover $g\colon \bar Y\to Y$, which is smooth since $\bar Y$ and $Y$ are smooth by assumption. So we can find $\bar t_1,\dots \bar t_{n-2}\in B\subset \bar B$ such that $\bar x, \bar t_1,\dots \bar t_{n-2}$ are local parameters on  $\bar Y$ near $Q$. Lifting all these elements to $\bar A$, we obtain local parameters $x,y, t_1,\dots t_{n-2}$ defining a map to $\A^n$ that is \'etale near $Q$. By Lemma \ref{lem: open}, we may find an open affine neighbourhood $U_P$ of $P$ such that  this map restricts to an \'etale map   $\phi_P\colon U_P\to \A^n$. It is immediate to check that diagram \eqref{eq: double-diag} is commutative, consists of two cartesian diagrams and the vertical arrows are \'etale.
 \smallskip

Next consider case $(b)$.  In this case $\bar x$ does not vanish at $Q_1$, $Q_2$. %so by Lemma \ref{lem: open} up to shrinking $X$  we may assume that  $\bar x$ is invertible.  
We claim that we may assume that $\bar x$ has nonzero differential at $Q_1$ and $Q_2$. Indeed, if this is not the case then the differential of $\bar x$ vanishes at both $Q_1$ and $Q_2$, because $\bar x$ is antiinvariant under the  the involution $\iota$ induced by $g$ and $\iota$ switches $Q_1$ and $Q_2$. So it is enough to multiply  $\bar x$ by a nonzero element $u\in B\subset \bar B$  that does not vanish and has nonzero  differential at $P$, and  possibly shrink $X$ again, so that  $u$ is a unit of $B$.  Finally we choose $\bar t_1,\dots \bar t_{n-2}$ in $B$ such that $\bar x^2, \bar t_1,\dots \bar t_{n-2}$ are local parameters on $Y$ at $Q$. It follows that $\bar x,\bar t_1,\dots \bar t_{n-2}$ are local parameters on $\bar Y$ at $Q_1$ and $Q_2$. One can now conclude the proof as in case $(a)$.

 %%%%%%%%%%%%%%%%%%%%%%%%%%%%%%%%%

  \subsection{Proof of $(i)\Rightarrow (ii)$}\label{ssec: A2}
  
Let  $X$ be a   variety over  the  algebraically closed field $\K$ of characteristic $\ne 2$. 
  We 
  denote by $f\colon \bar X\to X$ the normalization morphism. 
  Consider the exact sequence:
  $$0\to \OO_X\to f_*\OO_{\bar X}\to \mathcal Q\to 0;$$
  the conductor $\mathcal I \subset \OO_X$ is the ideal sheaf  $\Ann(\mathcal Q)$ and  is the largest ideal sheaf  of $\OO_X$  of the form  $ f_*\bar {\mathcal I}$ for some ideal sheaf $\bar{\mathcal I}$ of $\OO_{\bar X}$. We denote by $Y$, $\bar Y$ the zero scheme of $\mathcal I$, $\bar{\mathcal I}$  respectively. 
  By definition,  $Y_{\red}$ is precisely the set of non-normal points of $X$.  
  
    \begin{lem}\label{lem: base-change}
   Let $U$, $V$ be $\K$-varieties, let $f_U\colon \bar U\to U$ and $f_V\colon \bar V\to V$ be the normalization maps and let $\mathcal I_U$, $\mathcal I_V$ be the conductors of $U$ and  $V$. If $\phi\colon U\to V$ is  an \'etale morphism, then:   
  \begin{enumerate}
  \item  the  following diagram is cartesian: 
  \begin{equation}\label{diag: cartesian}
  \begin{CD}  \bar U @>{f_U}>>U\\
 @V{\bar \phi}VV @V{\phi}VV\\
 \bar V @>{f_V}>> V
  \end{CD}
  \end{equation}
  \item  $\phi^*\mathcal  I_V=\mathcal I_U$. 
  \end{enumerate}
  \end{lem}
  \begin{proof} (i) Consider the cartesian diagram:
  \[
  \begin{CD}  U' @>>>U\\
 @VVV @V{\phi}VV\\
 \bar V @>{f_V}>> V
  \end{CD}
  \]
  The morphism $U'\to \bar V$ is \'etale, so $U'$ is normal, since $\bar V$ is (\cite[\href{https://stacks.math.columbia.edu/tag/025P}{Tag 025P}]{stacks}). The morphism $U'\to U$ is finite and birational, since $f_V$ is, so there is a unique  isomorphism $U'\cong \bar U$ over $U$, and  via this identification the diagram above coincides with \eqref{diag: cartesian}. 
  \medskip
  
  (ii) Since cohomology commutes with flat base extension (\cite[Prop.~III.9.3]{Hartshorne}), the cartesian diagram \eqref{diag: cartesian}  gives  a natural isomorphism $\phi^*{f_V}_*\OO_{\bar V}\to {f_U}_*\bar \phi^*\OO_{\bar V}={f_U}_*\OO_{\bar U}$. So  we have a natural map $\phi^*\mathcal I_V\to \mathcal \mathcal I_U$; the fact that this map is an isomorphism can be checked by localizing and passing to completions, so it follows from the fact that the \'etale map $\phi$ induces an isomorphism on the completions of the local rings.
 
  \end{proof}

  \begin{lem}\label{lem: Xbar-smooth}
   If $X$ is a semi-smooth $
   \K$-variety, then:
  \begin{enumerate}
  \item  $\bar X$ is smooth;
  \item  $Y\subset X $ and $\bar Y\subset \bar X$ are smooth divisors;
  \item $f$ induces a finite degree 2 map $g\colon \bar Y\to Y$;
  \item $X$ is seminormal.
  \end{enumerate}
  \end{lem}
  \begin{proof} 
  Claims (i),(ii), (iii)  are local in the \'etale topology by  Lemma \ref{lem: base-change} and 
  are easily seen to hold  for $P_n= \Spec\K[u,v,w]/(u^2-v^2w)\times\A^{n-2}$.
  
Claim (iv) follows from  \cite[Prop.~I.7.2.5]{kollar-rational}  because of (ii).
    \end{proof}

\begin{proof}[Conclusion of proof of $(i)\Rightarrow (ii)$]
Let $f\colon \bar X\to X$ be the normalization and let $\bar Y\subset \bar X$ and $Y\subset X$ be the subschemes defined by the conductor ideal. By Lemma \ref{lem: Xbar-smooth}, $\bar X$ is smooth and $\bar Y$ and $Y$ are smooth of codimension 1. In addition, $\bar X$ and $Y$ satisfy condition (AF), since $X$ does, so we can consider the pushout scheme $X':=\bar X\sqcup_{\bar Y}Y$. By the universal property of the  pushout, there is  a birational morphism  $X'\to X$, which is the weak normalization of $X$ (cf.~\cite[Example 5]{kollar-quotients}). Since $X$ is semi-normal by Lemma \ref{lem: Xbar-smooth} and $X'$ is also semi-normal  (\cite[Prop.~7.2.3)]{kollar-rational}, the map $X'\to X$ is an isomorphism. 
 \end{proof}

  \vspace{.5cm}
  \hfill\break
Barbara Fantechi \\
SISSA  \\
 Via Bonomea 265, I-34136 Trieste (Italy)\\
{\tt fantechi@sissa.it}

  \vspace{.2cm}
   \hfill\break
Marco Franciosi\\
Dipartimento di Matematica, Universit\`a di Pisa\\
Largo B.~Pontecorvo 5,  I-56127 Pisa (Italy)\\
{\tt marco.franciosi@unipi.it}

\vspace{.2cm}

 \hfill\break
Rita Pardini\\
Dipartimento di Matematica, Universit\`a di Pisa\\
Largo B.~Pontecorvo 5,  I-56127 Pisa (Italy)\\
{\tt rita.pardini@unipi.it}


\begin{thebibliography}{ABCDE}
   
 \bibitem[AP12]{rita-valery} V.~Alexeev, R.~Pardini,
{\em Non-normal abelian covers}, 
Compos. Math. 
  \textbf{148}(4), 2012, 1051--1084. 
      
      \bibitem[FFP20]{FFP19b} B.~Fantechi, M.~Franciosi, R. Pardini, {\em  Smoothing  semi-smooth stable Godeaux surfaces}, https://arxiv.org/abs/2105.00786.
      \bibitem[FFR19]{ruddat} S.~Felten, M.~Filip, H.~Ruddat, {\em Smoothing toroidal crossing space}, 1--45. https://arxiv.org/abs/1908.11235
      \bibitem[Fe03]{Ferrand} D.~Ferrand, {\em Conducteur, descente et pincement},  Bull. Soc. Math. France 131 (2003), no. 4, 553--585.
   \bibitem[FPR18]{FPR18} M.~Franciosi, R. Pardini, S.~Rollenske, {\em Gorenstein stable Godeaux  surfaces}, Selecta Mathematica 24 (4) (2018), 3349-3379. DOI:10.1007/s00029-017-0342-6.

\bibitem[Fr83]{Friedman} R. Friedman, {\em Global smoothings of varieties with normal crossings},  Ann. Math. {\bf 118} (1983), 75--114.

\bibitem[Fu84]{Fulton} W.~Fulton, {Intersection Theory}, Ergebnisse der Math.\ und i.\ G., Springer-Verlag (1984). 


\bibitem[Ha77]{Hartshorne}  R.~Hartshorne, Algebraic geometry,  Graduate Texts in Mathematics, No. 52. Springer-Verlag, New York-Heidelberg (1977). xvi+496


\bibitem[KSB]{KSB} J.~Koll\'ar, N.~Shepherd-Barron, {\em Threefolds and deformations of surface singularities},
Invent. Math., {\bf 91}(2) (1988), 299--338.
\bibitem[Ko11]{kollar-quotients} J.~Koll\'ar, {\em Quotients by finite equivalence relations. With an appendix by Claudiu Raicu.}, Math. Sci. Res. Inst. Publ., 59, Current developments in algebraic geometry, 227--256, Cambridge Univ. Press, Cambridge  (2012). 
\bibitem[Ko13]{kollar-sings} J.~Koll\'ar,  Singularities of the minimal model program.
With a collaboration of S. Kov\'acs. Cambridge Tracts in Mathematics, 200. Cambridge University Press, Cambridge (2013). 
\bibitem[Ko96]{kollar-rational} J.~Koll\'ar, Rational curves on algebraic varieties.
Ergebnisse der Mathematik und ihrer Grenzgebiete. 3. Folge.  Springer-Verlag, Berlin (1996).

\bibitem[Pa91]{rita-abel} R. Pardini, {\em Abelian covers of algebraic varieties},   J.
reine angew. Math.  {\bf 417} (1991), 191--213.
\bibitem[Qu70]{Quillen} D.~Quillen, {\em On the (co-)homology of commutative rings}, Proc. Symp. Pure Mat., XVII, (1970) 65--87. 
\bibitem[Ri15]{Riz} A.Rizzardo, 
{\em Adjoints to a Fourier-Mukai functor}, Advances in Mathematics {\bf 322} (2017), 83--96.

\bibitem[Sch71]{Schlessinger} M. Schlessinger, {\em  Rigidity of quotient singularities}, Invent. Math. {\bf 14}, 17--26 (1971). 


\bibitem[St-Pr]{stacks} The Stacks Project Authors, {\em Stacks Project}, (2018)
https://stacks.math.columbia.edu 
\bibitem[Tzi10]{tziolas2010}  N.~Tziolas, 
{\em Smoothings of schemes with nonisolated singularities},
Michigan Mathematical Journal  {\bf 59}    (2010),    25--84.
  \end{thebibliography}
  \end{document}